\documentclass[12pt]{article}

\usepackage{amsfonts}
\usepackage{amsmath}
\usepackage{amssymb}
\usepackage{amscd}
\usepackage{amsthm}
\usepackage{indentfirst}

\newtheorem{thm}{Theorem}[section]
\newtheorem{lemma}[thm]{Lemma}
\newtheorem{prop}[thm]{Proposition}
\newtheorem{cor}[thm]{Corollary}

\newtheorem{rem}[thm]{Remark}

\newtheorem{problem}[thm]{Problem}
\newtheorem{question}[thm]{Question}

\newcommand{\R}{{\mathbb{R}}}

\newcommand{\N}{{\mathbb{N}}}

\newcommand{\cF}{{\mathcal{F}}}

\def\id{{1\hskip-2.5pt{\rm l}}}

\newcommand{\tF}{{\widetilde{F}}}
\newcommand{\tG}{{\widetilde{G}}}

\newcommand{\osc}{{\hbox{\rm osc}}}

\newcommand{\Ham}{{\hbox{\it Ham\,}}}
\newcommand{\Hamc}{{{\it Ham^c}}}

\newcommand{\tHam}{\widetilde{\hbox{\it Ham}\, }}

\newcommand{\supp}{{\it supp\,}}

\newcommand{\Qed}{\hfill \qedsymbol \medskip}

\begin{document}

\title{$C^0$-rigidity of the double Poisson bracket \\
}

\renewcommand{\thefootnote}{\alph{footnote}}

\author{\textsc Michael Entov$^{a}$,\ Leonid
Polterovich$^{b}$}

\footnotetext[1]{Partially supported by the Israel Science
Foundation grant $\#$ 881/06.} \footnotetext[2]{Partially
supported by the Israel Science Foundation grant $\#$ 509/07.}

\date{\today}

\maketitle

\begin{abstract}
\noindent The paper is devoted to function theory on symplectic
manifolds.  We study a natural class of functionals involving the
double Poisson brackets from the viewpoint of their robustness
properties with respect to small perturbations in the uniform norm.
We observe a hierarchy of such robustness properties.  The methods
involve Hofer's geometry on the symplectic side and
Landau-Hadamard-Kolmogorov inequalities on the function-theoretic
side.
\end{abstract}

\vfill\eject

\tableofcontents

\renewcommand{\thefootnote}{\arabic{footnote}}
\vfil \eject

\section{Introduction and main results}
\label{sec-intro}

A mainstream topic of modern symplectic topology is the study of
rigidity properties of subsets and Hamiltonian diffeomorphisms of
symplectic manifolds. A number of recent developments (see
\cite{Car-Vit,EPZ,Z,Hum,EP-Poisson1,Buh}) show that there is
another manifestation of symplectic rigidity which takes place on
function spaces associated to a symplectic manifold. This circle
of problems, which we call {\it function theory on symplectic
manifolds}, lies in the focus of the present paper.

\subsection{A dichotomy}
Let $(M, \omega)$ be a connected symplectic manifold (open or
closed). Denote by $C^\infty_c (M)$ the space of smooth compactly
supported functions on $M$ equipped with the Poisson bracket
$\{F,G\}$.  Write $\|\cdot \|$ for the standard {\it uniform norm}
(also called the {\it $C^0$-norm}) on it: $ \| F\| := \max_{x\in
M} |F(x)|$. As it was shown in \cite{EP-Poisson1} the functionals
$(F,G) \mapsto \max \{F,G\}$ and $(F,G) \mapsto -\min \{F,G\}$ are
lower semicontinuous on $C^\infty_c (M) \times C^\infty_c (M)$
with respect to the uniform norm. Thus, even though the Poisson
bracket of a pair of functions is defined via their first
derivatives, it exhibits a robust behavior under $C^0$-small
perturbations.

In the present paper we explore the double Poisson bracket.  We
focus on non-negative\footnote{In the case of a closed $M$ the
mean value of a Poisson bracket of two functions on $M$ is zero,
hence its minimum is non-positive and its maximum is non-negative.
In the open case the latter properties are obvious since we work
with compactly supported functions.} functionals $\Phi^v(F,G)$ of
the form
$$\Phi^v(F,G)= v_1 \cdot \max\{\{F,G\},F\} - v_2 \cdot \min\{\{F,G\},F\}$$
$$+v_3 \cdot\max\{\{F,G\},G\} -v_4 \cdot\min\{\{F,G\},G\}\;,$$
where $v =(v_1,v_2,v_3,v_4) \in \R^4$ is a non-zero vector with
non-negative entries. Interestingly enough, these functionals
exhibit different patterns of behavior depending on $v$. To
highlight the phenomenon, given a non-negative functional
$\Phi(F,G)$, we form a new functional
\begin{equation}\label{eq-bar}\overline{\Phi}(F,G):=
\liminf_{F',G'\stackrel{C^0}{\longrightarrow}F,G}
\Phi(F',G')\;\end{equation} and introduce the following
terminology: A functional $\Phi$ is called {\it weakly robust} if
$\overline{\Phi}(F,G)> 0$ whenever $\Phi(F,G) > 0$. With this
language $\Phi$ is {\it lower semicontinuous } if $\overline{\Phi}
=\Phi$.

\medskip
\noindent \begin{thm} [Dichotomy] \label{thm-main-1}
\item[{(i)}]  If either $v_3=v_4 =0$ {\bf or}
$v_1=v_2=0$,  $\Phi^v$ is weakly robust but not lower
semicontinuous;
\item[{(ii)}] If at least one of $v_1,v_2$ is positive
{\bf and} at least one of $v_3,v_4$ is positive, $\Phi^v$ is lower
semicontinuous.
\end{thm}

\medskip
\noindent In order to verify the  the failure of lower
semicontinuity in (i) we proceed as follows. First we consider the
case when $(M,\omega) = (\R^2(p,q), dp\wedge dq)$. We put $F=F(p),
G=G(q)$, and $$F_N(p,q) = F(p) + \frac{1}{N}a(q) \sin NF(p).$$
Note that $F_N\stackrel{C^0}{\longrightarrow}F$ as $N \to \infty$.
We construct a rather explicit example of the functions $F,G,a$ so
that for sufficiently large $N$
$$\max (\pm \{\{F_N,G_N\},F_N\}) < 0.99\cdot \max (\pm
\{\{F,G\},F\}).$$ This example can be implanted into arbitrary
symplectic manifolds which eventually yields the desired result (see
Section~\ref{sec-oscFGG-example} and \ref{sec-end} for the details).
The remaining statements of the Dichotomy Theorem deserve a more
detailed discussion.

\subsection{Landau-Hadamard inequality for the Poisson
bra\-cket} Let $(M,\omega)$ be a connected symplectic manifold of
dimension $2n$. Put $\osc\, F := \max F - \min F$. Note that,
given $F,G \in C^{\infty}_c(M)$, one has \break $\int_M \{F,G\}
\;\omega^n =0$, and thus $\osc\;\{F,G\} \geq ||\{F,G\}||$.

The next inequality is a variant of the classical Landau-Hadamard
inequality \cite{Landau, Hadamard}, cf. \cite{Mitrinovic-et-al},
in the context of the Poisson brackets:
\begin{prop}
\label{thm-double-bracket}Assume $F, G\in C^\infty_c (M)$ and
$G\not\equiv 0$. Then \begin{equation}
\label{eqn-double-bracket-osc} \max \{\{F,G\}, F\} \geq \frac{
||\{ F, G\}||^2 }{2 \cdot \osc\, G}.
\end{equation}
\end{prop}
\noindent The proof is given in Section~\ref{sec-LH} below. Since,
as we mentioned at the beginning of the paper, the functional $||
\{F,G\}||$ is lower semicontinuous \cite{EP-Poisson1}, we readily
get that
\begin{equation}
\label{eqn-double-bracket-liminf-osc}
\liminf_{F',G'\stackrel{C^0}{\longrightarrow}F,G} \max \{\{F',G'\},
F'\} \geq \frac{|| \{ F, G\}||^2 }{2\cdot \osc\, G}>0\;,
\end{equation} which implies that the functional $\max
\{\{F,G\},F\}$ is weakly robust. After elementary algebraic
manipulations this yields weak robustness in Theorem
\ref{thm-main-1}(i), see Section~\ref{sec-end}.

\subsection{The convergence rate }
In the case when a functional $\Phi$ is lower semicontinuous we
investigate the ``convergence rate" in the limit \eqref{eq-bar}
(cf. \cite{EPZ}). For every $\epsilon >0$ put
$$\overline{\Phi}_{\epsilon}(F,G)= \inf_{||F-F'||\leq \epsilon, ||G'-G|| \leq
\epsilon} \Phi(F',G')\;.$$ We are interested in  upper bounds for
the difference $$ \Phi(F,G) - \overline{\Phi}_{\epsilon}(F,G)$$ in
terms of $F,G$ and $\epsilon$ as $\epsilon \to 0$. L.Buhovsky
\cite{Buh} discovered such a (sharp!) upper bound for the functional
$\Phi(F,G) = \max\{F,G\}$. To state this result, we put
\begin{equation}\label{eq-psi}
\Psi(F,G):= ||\{\{\{F,G\},F\},F\} + \{\{\{F,G\},G\},G\}||\;.
\end{equation} One can show that $\Psi(F,G) >0$ provided
$\{F,G\}\neq 0$ (see Corollary~\ref{cor-FGFF}  below). With this
notation, Buhovsky derived the following $2/3$-law: There exists
$\epsilon_0(F,G)>0$ such that for any $0< \epsilon <
\epsilon_0(F,G)$
\begin{equation}\label{eq-main}\Phi(F,G) - \overline{\Phi}_{\epsilon}(F,G) \leq C \cdot
\Psi(F,G)^{\frac13}\epsilon^{\frac23}\;,
\end{equation} where $C>0$ is a numerical constant. Furthermore,
Buhovsky showed that his estimate captures sharp asymptotics in
$\epsilon$.  Buhovsky's proof of \eqref{eq-main} is based on an
ingenious application of the energy-capacity inequality. In
Section~\ref{sec-buhlaw} we reprove \eqref{eq-main} by our
methods.

\medskip
\noindent For the case of the double bracket we have the following
estimate of the convergence rate:

\medskip
\noindent
\begin{thm}[Convergence rate]\label{thm-rate} Let
$$\Phi(F,G) = \max\{\{F,G\},F\} +
\max\{\{F,G\},G\}\;.$$ Then for every $\epsilon >0 $
\begin{equation}\label{eq-main-doub}\Phi(F,G) - \overline{\Phi}_{\epsilon}(F,G)\leq C(F,G)\cdot\epsilon^{\frac13}\;,
\end{equation}
where $C(F,G)$ is a positive constant depending on $F$ and $G$.
\end{thm}

\medskip
\noindent The proof is given in Section~\ref{sec-rate-double}
below. It is unclear to us whether the theorem above can be
improved:

\medskip
\noindent \begin{question}\label{quest-power} Is the power law
$\epsilon^{\frac13}$ in inequality \eqref{eq-main-doub}
asymptotically sharp as $\epsilon \to 0$?
\end{question}

\medskip
\noindent We return to this question in Remark~\ref{rem-sharpness}
below.

\medskip
\noindent As an immediate consequence of inequality
\eqref{eq-main-doub} we get that the functional $\Phi$ is lower
semicontinuous. With a little extra work, we deduce from this
Theorem~\ref{thm-main-1}(ii), see Section~\ref{sec-end}.

\medskip
\noindent For the proof of Theorem~\ref{thm-rate}  we use the
approach initiated in \cite{EP-Poisson1} which is based on the
following ingredient from ``hard" symplectic topology: Denote by
$\Hamc (M)$ the group of Hamiltonian diffeomorphisms of $M$
generated by Hamiltonian flows with compact support. Then
sufficiently small segments of one-parameter subgroups of the
group $\Hamc (M)$ of Hamiltonian diffeomorphisms of $M$ minimize
the ``positive part of the Hofer length" among all paths on the
group in their homotopy class with fixed end points. This was
proved  by D.McDuff in \cite[Proposition 1.5]{McD-variants} for
closed manifolds and in \cite[Proposition 1.7]{McD-monodromy} for
open ones; see also \cite{Bialy-Pol}, \cite{Lal-McD}, \cite{En},
\cite{McD-Slim}, \cite{KL}, \cite{Oh} for related results in this
direction. In fact, our method readily generalizes to any (in
general, ``infinite-dimensional") Lie group equipped with a
bi-invariant (Finsler) semi-norm, provided sufficiently short
segments of 1-parameter subgroups are minimal geodesics. It would
be interesting to formalize this remark and to find new
significant examples.

The results discussed above can be viewed as a symplectic
counter-part of the following classical problem of approximation
theory: find the best uniform approximation of a given function
(say, of a periodic function of one real variable) by functions
with given bounds on derivatives. This problem was solved in the
1960s, see Sections 6.2.1 and 7.2.3 of Korneichuk's book
\cite{Korneichuk} and the references therein. For instance, one
can extract from Korneichuk's results that the functional taking a
smooth periodic function $u$ on $\R$ to the uniform norm of its
derivative is lower semicontinuous in the uniform norm and obeys
the $2/3$-law (see \cite{Buh} for a direct proof and some
generalizations). It would be interesting to explore further the
connection between approximation theory and function theory on
symplectic manifolds.

\medskip
\noindent{\sc Organization of the paper.} In Section~\ref{sec-LH}
we prove a version of the Landau-Hadamard inequality for the
double Poisson bracket and thus complete the proof of weak
robustness of the functional $\max \{\{F,G\},F\}$. In Section
\ref{sec-Hofer}, after recalling some preliminaries on Hofer's
geometry, we give a new proof of Buhovsky's 2/3-law
\eqref{eq-main} for the ordinary Poisson bracket and prove the
Convergence Rate Theorem for the double bracket. In Section
\ref{sec-oscFGG-example} we construct an example which in
particular shows that the functional $\osc\,\{\{F,G\},F\}$ is not
lower semicontinuous. The proof of the Dichotomy Theorem is
completed in Section~\ref{sec-end}. Finally, in Section
\ref{sec-disc} we present some generalizations of our results on
weak robustness to higher iterated Poisson brackets and formulate
open problems.

\medskip
\noindent {\bf Convention on the Poisson bracket:} Our convention
concerning  the Poisson bracket is as follows: Given a Hamiltonian
function $G\in C^\infty_c (M)$, its Hamiltonian vector field,
denoted by $sgrad\,G$, is defined by the condition $dG (\cdot) :=
\omega (\cdot,sgrad\, G )$.  The Poisson bracket of a pair of
functions $F,G\in C^\infty_c (M)$ is then defined by
$$\{ F,G\} := \omega (sgrad\,G, sgrad\,F)= dF (sgrad\, G) = \frac{d}{dt}\Big{|}_{t=0} F \circ
g_t\;,$$ where $g_t$ is the Hamiltonian flow generated by $G$.

\section{The Landau-Hadamard inequality}\label{sec-LH}

We shall need the following version of the classical Landau-Hadamard
inequality.

\medskip
\noindent
\begin{prop}\label{ineq-lh}
Let $u$ be a non-constant twice differentiable function on $\R$
which is bounded with its two derivatives. Assume that $|u'|$
attains its maximal value. Then
\begin{equation}\label{ineq-hadam}
\sup u'' \geq \frac{||u'||^2}{2\cdot \osc\,u}\;.
\end{equation}
\end{prop}

\medskip
\noindent \begin{proof} Assume without loss of generality that $
||u'||=|u'(0)|$. Denote $A = \sup u''$. For every $t$
\begin{equation}\label{eq-vsp-had}
\osc\, u \geq u(0)-u(t) =  -tu'(0) - \int_0^t ds \int_0^s u''(z)\;dz
\geq -tu'(0)-At^2/2\;.
\end{equation}
Note that  $A> 0$, otherwise $u$ is either constant or unbounded.
Substituting $t = -u'(0)/A$ into \eqref{eq-vsp-had} we get
inequality \eqref{ineq-hadam}.
\end{proof}

\medskip
\noindent{\bf Proof of Proposition~\ref{thm-double-bracket}}: Take
a point $x \in M$ so that $|| \{F,G\}||= \{F,G\}(x)$. Denote by
$f_t$ the Hamiltonian flow of $F$ and put $u(t):= -G(f_t x)$. Then
$u'(t)= \{F,G\}(f_tx)$ and $u''(t) = \{\{F,G\},F\}(f_tx)$. Note
that
$$\sup u'' \leq \max \{\{F,G\},F\},\; || u'||=|| \{F,G\}||,\; \osc\; u \leq
\osc \,G\;.$$ Applying inequality \eqref{ineq-hadam} to $u$ we get
that
$$\max \{\{F,G\}, F\} \geq \frac{ ||\{ F, G\}||^2 }{2\cdot\osc\, G}\;,$$
as required. \qed

\medskip
\noindent
\begin{cor}\label{cor-FGFF} Let $F,G \in C^{\infty}_c(M)$ be a pair
of functions with $\{F,G\} \not\equiv 0$. Then
$$I(F,G):=\{\{\{F,G\},F\},F\} + \{\{\{F,G\},G\},G\} \not\equiv
0\;.$$
\end{cor}

\begin{proof}
We shall use the following classical identity which readily
follows from the definition of the Poisson bracket and the Stokes
formula:
$$\int_M \{P,Q\}R\;\omega^n= \int_M \{R,P\}Q\;\omega^n$$
for all functions $P,Q,R \in C^{\infty}_c(M)$. This implies that
$$\int_M I(F,G) \cdot \{F,G\} \;\omega^n = -\int_M
\{\{F,G\},F\}^2+\{\{F,G\},G\}^2 \;\omega^n\;.$$ If $I(F,G)\equiv 0$
we have that $\{\{F,G\},F\}\equiv 0$. By Proposition
\ref{thm-double-bracket} this yields $\{F,G\}=0$. \end{proof}

\section{Poisson bracket via Hofer's geometry}\label{sec-Hofer}

\subsection{Preliminaries}

\noindent Let $(M,\omega)$ be a connected symplectic manifold of
dimension $2n$. Write $\tHam (M)$ for the universal cover of the
group $\Ham(M)$ of Hamiltonian diffeomorphisms of $(M,\omega)$,
where the base point is chosen to be the identity map $\id$.
Denote by $\phi_H^t \in \tHam(M)$, $t\in\R$, the lift of the
Hamiltonian flow generated by a (time-dependent) Hamiltonian $H$.
We set $\phi_H :=\phi_H^1$ and say that $\phi_H$ {\it is generated
by} $H$. If $H$ is time-independent, we often abbreviate $\phi_H^t
= h_t$.

\noindent If $F$ is a function on $M$ and $\phi \in \tHam(M)$, we
(by a slight abuse of notation) write $F \circ \phi$ for the
composition of $F$ with the projection of $\phi$ to $\Ham(M)$.

\noindent For a time-dependent Hamiltonian $H(x,t)$ we  set $H_t =
H(\cdot, t)$.

\noindent Denote by $\cF$ the set of all the Hamiltonians $H$ on $M$
such that
\begin{itemize}
\item if $M$ is open, the union of supports of $H_t$, $t \in
[0,1]$, is compact; \item if $M$ is closed, $H_t$ has zero mean
for all $t$: $\int_M H_t \;\omega^{n} = 0$.
\end{itemize}

\noindent The group $\tHam(M)$ carries a conjugation-invariant
functional $\rho$ (called the ``positive part" of the Hofer's norm)
defined by
\[\rho(\phi) :=
\inf_H \int_0^1 \max_{x \in M} H(x,t)\;dt\;,
\]
where the infimum is taken over all (time-dependent) Hamiltonians
$H\in\cF$ generating $\phi$.

\noindent We shall often use the following well known properties
of the functional $\rho$ which readily follow from the definition.

\begin{prop}\label{prop-dob}
\begin{itemize}
\item[{(i)}] (conjugation invariance)
$\rho(\psi\phi\psi^{-1})=\rho(\phi)$ for all $\phi,\psi \in \tHam
(M)$. \item[{(ii)}] (triangle inequality) $ \rho(\phi\psi) \leq
\rho(\phi)+\rho(\psi)$ for all $\phi,\psi \in \tHam(M)$.
\item[{(iii)}] $\rho(\phi_F^{-1}\phi_G) \leq \int_0^1 \max(G_t
-F_t)\;dt\;$ for all $F,G \in \cF$.
\end{itemize}
\end{prop}

\medskip
\noindent Combining items (ii) and (iii) of the proposition, we
get that
\begin{equation}\label{ineq-diff}
|\rho(\phi_F)-\rho(\phi_G)| \leq
\max(\rho(\phi_F^{-1}\phi_G),\rho(\phi_G^{-1}\phi_F) )\leq
\int_0^1 ||F_t -G_t||\;dt\;.
\end{equation}

\medskip
\noindent Let us illustrate the conjugation invariance of $\rho$
and the triangle inequality by proving the following lemma which
will be useful in the sequel.

\medskip
\noindent {\bf Convention on commutators:} We write $[\phi,\psi]$
for the commutator \break $\phi\psi\phi^{-1}\psi^{-1}$ of elements
$\phi,\psi \in \tHam(M)$.

\medskip
\noindent
\begin{lemma}\label{lem-alg} For every elements $a,b,c,d \in
\tHam(M)$
\begin{equation}\label{eq-lem-alg}
|\rho([a,b])-\rho( [c,d])| \leq \rho(a^{-1}c) + \rho(c^{-1}a) +
\rho(b^{-1}d) + \rho(d^{-1}b)\;.
\end{equation}
\end{lemma}

\medskip
\noindent
\begin{proof} Set $e:=[a,b]^{-1}[c,d]$. Note that $e$ is conjugate to $fg$ with
$f= b^{-1}a^{-1}cd, g= c^{-1}d^{-1}ba$. In turn, $f$ is conjugate
to $(a^{-1}c)(db^{-1})$ and $g$ is conjugate $(d^{-1}b)(ac^{-1})$.
Using the conjugation invariance of $\rho$ and the triangle
inequality we conclude that $\rho(e)$ does not exceed the
right-hand side of \eqref{eq-lem-alg}. A similar analysis proves
the same bound for $\rho(e^{-1})$. But the left inequality in
\eqref{ineq-diff} shows that the left-hand side of
\eqref{eq-lem-alg} does not exceed $\max(\rho(e),\rho(e^{-1}))$,
and thus inequality \eqref{eq-lem-alg} follows.
\end{proof}

\medskip
\noindent The key Hofer-geometric ingredient used in the proofs
below is as follows (McDuff, \cite[Proposition1.5]{McD-variants}):
 for every {\it time-independent} function $H \in \cF$ there
exists $\delta > 0$ so that
\begin{equation}\label{eq-geodes}
\rho(\phi_{tH}) = t\cdot \max H  \;\;\forall t \in (0,\delta)\;.
\end{equation}

\subsection{Buhovsky's $2/3$ law }\label{sec-buhlaw}

As a warm up we prove formula \eqref{eq-main}. Put
$$P:=\{F,G\},\ \ I:=\{\{P,F\},F\}+\{\{P,G\},G\}.$$ Note that $\Psi=||I||$. Recall
from Corollary~\ref{cor-FGFF} that $\Psi>0$. Consider the
Hamiltonian flow
$$v_\tau:= \phi_{\tau(F+G)/2}f_{-\tau}g_{-\tau}f_\tau g_\tau \phi_{-\tau(F+G)/2}\;.$$
A lengthy but straightforward calculation (which we checked by the
slightly modified Maple-based Lie Tools Package software
\cite{Lie-software}) shows that the corresponding Hamiltonian
$V(\tau):=V(x,\tau)$ has expansion
\begin{equation}\label{eq-expansion}
V(\tau) = 2\tau P + \frac{\tau^3}{6}I + O(\tau^4)\;.
\end{equation}
Let $t$ be a small parameter, and $\tau \in [0,1]$ be the time
variable. Consider the flow $v_{t\sqrt{\tau}}$ whose time one map
equals $v_t$. By \eqref{eq-expansion} this flow is generated by
the Hamiltonian
$$(2\sqrt{\tau})^{-1}tV(t\sqrt{\tau})= t^2P + R\;,$$
where $$R= \frac{t^4\tau}{12}I+ O(t^5)\;.$$ Fix $\delta > 0$. By
\eqref{ineq-diff} there exists $t_0 >0$ so that for every $ 0<t<t_0$
\begin{equation}\label{eq-buh-vsp-1}
|\rho(v_t)- \rho(\phi_{t^2P})| \leq  \frac{1+\delta}{24}\Psi \cdot
t^4\;.\end{equation} Decreasing if necessary $t_0$ we have from
\eqref{eq-geodes} that $\rho(\phi_{t^2P})=t^2 \max P$ for
$0<t<t_0$. Furthermore, $v_t$ is conjugate to $[f_t,g_t]$ and
hence $\rho(v_t)= \rho([f_t,g_t])$. Thus \eqref{eq-buh-vsp-1}
yields
\begin{equation}\label{eq-fund}
\rho([f_t,g_t]) \geq t^2\max\{F,G\} - \frac{1+\delta}{24}\Psi\cdot
t^4\;,
\end{equation}
for $0<t<t_0$.

Now put $K_{\delta} = \frac{1+\delta}{24}\Psi$. Choose a small
positive $\epsilon$ so that $(4\epsilon\cdot
K_{\delta}^{-1})^{1/3} \leq t_0$.
 Let $F',G'$ be functions with
$||F-F'||\leq \epsilon, ||G'-G|| \leq \epsilon$, and let
$f'_t,g'_t$ be the corresponding Hamiltonian flows. By an
elementary ODE Lemma 2.2 of \cite{EP-Poisson1}
\begin{equation}\label{eq-fund-1}
\rho([f'_t,g'_t]) \leq t^2\max\{F',G'\}\;,\;\;\text{for all}\;\;
t\;.
\end{equation}

Next we claim that
\begin{equation}\label{eq-straight}
|\rho([f'_t,g'_t])-\rho([f_t,g_t])| \leq 8 \epsilon t\;\;\text{for
all}\;\; t\;.
\end{equation}
When $M$ is an open manifold, this readily follows from
Lemma~\ref{lem-alg} and Proposition~\ref{prop-dob}(iii), and
actually we get the numerical constant $4$ instead of $8$ in
\eqref{eq-straight}. When $M$ is a closed manifold, we have to be
a bit more careful since the inequality in
Proposition~\ref{prop-dob}(iii) holds only for normalized
Hamiltonians. For any function $H \in C^{\infty}(M)$ define its
normalization by
$$H_{norm}:= H - \frac{1}{\text{Volume}(M)} \cdot \int_M H \;
\omega^n\;.$$ Clearly, for any two functions $H,H'$ we have
$$||H_{norm}-H'_{norm}|| \leq 2||H-H'||\;.$$
With this in mind, we again use Lemma~\ref{lem-alg} and apply
Proposition~\ref{prop-dob}(iii) to the normalizations of the
functions $tF,tF',tG,tG'$. This readily yields inequality
\eqref{eq-straight}.

Combining inequalities  \eqref{eq-fund},\eqref{eq-fund-1} and
\eqref{eq-straight} one gets that for all $t$, $0<t<t_0$,
$$\max\{F',G'\} \geq \max\{F,G\} -(8\epsilon t^{-1} + K_{\delta} t^2)\;,$$
and hence
$$\max\{F,G\}-\max\{F',G'\} \leq 8\epsilon t^{-1} + K_{\delta} t^2.$$
As a function of $t$, $t>0$ (for fixed $\epsilon$ and $K_\delta$),
the right-hand side reaches its minimum at $t = (4\epsilon\cdot
K_{\delta}^{-1})^{1/3}$ which, by our choice of $\epsilon$,
belongs to the interval $(0, t_0)$. Hence, we may substitute this
$t$ in the right-hand side which yields
$$\max\{F,G\}-\max\{F',G'\} \leq C (1+\delta)^{1/3}
\Psi^{1/3}\epsilon^{2/3}\;,$$ where $C$ is a numerical constant.
This immediately yields the desired formula \eqref{eq-main}. \qed

\subsection{Convergence rate for the double
bracket}\label{sec-rate-double} In this section we prove Theorem
\ref{thm-rate}. Throughout the proof we use notation
$$\theta(F,G) = [\phi_{-F} \phi_{-G}, \phi_{F+G}]\;.$$

\begin{lemma}
\label{lemma-1} For all $F,G \in C^{\infty}_c(M)$
$$\rho(\theta(F,G)) \leq (\max \{ \{
F,G\},F\} + \max \{\{ F,G\}, G\} ) / 2$$
\end{lemma}

\begin{proof}
The diffeomorphism $\theta(F,G)$ can be generated by a Hamiltonian
flow $$[\phi_{-F} \phi_{-G}, \phi_{\tau(F+G)}]\;.$$ The
corresponding Hamiltonian is given by
$$H (x,\tau) = (F+G)\phi_G\phi_F -
(F+G)\phi_{-F}\phi_{-G}\phi_{-\tau (F+G)} \phi_G\phi_F.$$ Clearly,
$$\max_x H (x,\tau) = \max_x \big( F+G -
(F+G)\phi_{-F}\phi_{-G}\big) = \max_x \big( G + F\phi_G - G\phi_{-F}
- F\big).$$ Denote
$$Y  := G + F\phi_G -
G\phi_{-F} - F.$$ Then
\begin{equation}
\label{eqn-1} \rho (\theta(F,G)) \leq \max_x H = \max_x Y.
\end{equation}
On the other hand, it is easy to see that
$$Y(x) = \int_0^1 (\{ F, G\}\phi_{uG} + \{G,F\}\phi_{-uF}) du =$$
$$=\int_0^1\int_0^u (\{ \{ F, G\}, G\} \phi_{wG} + \{ \{ G,F\},
-F\}\phi_{-wF}) dw du \leq$$ $$\leq (\max_x \{ \{ F, G\}, G\} +
\max_x \{ \{ F, G\}, F\}) \int_0^1 \int_0^u dw du =$$ $$=(\max_x \{
\{ F, G\}, G\} + \max_x \{ \{ F, G\}, F\})/2.$$ Hence
$$\max_x Y \leq (\max_x \{ \{ F, G\}, G\} +
\max_x \{ \{ F, G\}, F\})/2.$$ By (\ref{eqn-1}) this implies
$$\rho (\theta(F,G)) \leq (\max_x \{
F,G\},G\} + \max_x \{ F,G\}, F\} ) / 2,$$ as needed.
\end{proof}

\medskip

Denote $A:= \frac12 \{ \{ F,G\},F\}$ and $B:=  \frac12 \{ \{
F,G\},G\}$.   Consider the flow $$v_{\tau} =
\phi_{\tau(F-G)/6}\circ \theta(\tau F,\tau G)\circ
\phi^{-1}_{\tau(F-G)/6}\;.$$ A lengthy but straightforward
calculation (which we checked by the slightly modified Maple-based
Lie Tools Package software \cite{Lie-software}) shows that the
corresponding Hamiltonian $V(\tau):=V(x,\tau)$ has expansion
\begin{equation}\label{eq-expansion-1}
V(\tau) = 3 \tau^2 (A+B) + \tau^4 Q  + O(\tau^5)\;,
\end{equation}
where $Q$ is a Lie polynomial of $F$ and $G$ whose monomials are
4-times-iterated Poisson brackets. Let $s,t$ be small parameters.
Replacing $F \to sF, G \to tG$ and making the change of time $\tau
\to \tau^{\frac13}$ we get that the element
$$u_{s,t}:= \phi_{(sF-tG)/6}\circ \theta(sF,tG)\circ \phi^{-1}_{(sF-tG)/6}\;$$
is generated by Hamiltonian $s^2t A + st^2B + R$ with
$$||R|| \leq E\cdot \sum_{i=1}^4 s^i t^{5-i}\;,$$
where $E$ is a constant depending on $F$ and $G$. Applying
\eqref{ineq-diff} and \eqref{eq-geodes} (here hard symplectic
topology enters the play) and taking into account that $u_{s,t}$ is
conjugate to $\theta(sF,tG)$ we get that for sufficiently small
$s,t>0$
\begin{equation}\label{eq-rate-vsp-1}
\rho(\theta(sF,tG))=  \rho(u_{s,t}) \geq \max(s^2tA +st^2B) - E\cdot
\sum_{i=1}^4 s^i t^{5-i}\;.
\end{equation}

\medskip

Take any $F',G'$ with $||F-F'||\leq \epsilon, ||G'-G|| \leq
\epsilon$. Put  $$A':= \frac12 \{ \{ F',G'\},F'\}, \ \ B':=
\frac12 \{ \{ F',G'\},G'\}.$$ By Lemma~\ref{lemma-1},
\begin{equation}\label{eq-d-vsp-1}
\rho(\theta(sF',tG') \leq s^2t \max A' + st^2 \max B'\;.
\end{equation}
Furthermore,
\begin{equation} \label{eq-d-vsp-2}
|\rho(\theta(sF,tG))-\rho(\theta(sF',tG'))| \leq 16\epsilon
(s+t)\;.
\end{equation}  The proof of this inequality is similar to
the proof of inequality \eqref{eq-straight} above: it readily
follows from Lemma~\ref{lem-alg} and
Proposition~\ref{prop-dob}(iii). We omit the details.

Combining inequalities \eqref{eq-rate-vsp-1}, \eqref{eq-d-vsp-1}
and \eqref{eq-d-vsp-2} we get that for all sufficiently small
$s,t>0$
\begin{equation}\label{eq-rate-vsp-2}
s^2t\max A' + st^2 \max B' \geq \max (s^2tA +st^2B) -E\cdot
\sum_{i=1}^4 s^i t^{5-i}-16\epsilon(s+t)\;.
\end{equation}
Put
\begin{equation}\label{eq-rate-delta}
2\Delta:= \max A + \max B - \max A' - \max B'\;.
\end{equation}
We have to find an upper bound on $\Delta$ assuming that $\Delta >
0$. Without loss of generality assume that $\max A - \max A' \geq
\Delta$.

Since $\Delta > 0$,  \eqref{eq-rate-delta} yields
\begin{equation}\label{eq-rate-vsp-3}
\max B' \leq ||A||+||B||\;.
\end{equation}
Further, $$\max (s^2tA +st^2B) \geq s^2t \max A - st^2 ||B||\;.$$
Substituting these inequalities into \eqref{eq-rate-vsp-2} we get
that
$$s^2t\max A'+st^2(||A||+||B||) \geq s^2t \max A - st^2 ||B||-E\cdot \sum_{i=1}^4 s^i t^{5-i}-16\epsilon(s+t)\;.$$
Since $\Delta\leq \max A - \max A'$, we conclude that for
sufficiently small $s$ and $t$
$$\Delta \leq Rs^{-2}t^{-1},$$
 where $R= (||A||+2||B||) st^2+E\cdot \sum_{i=1}^4 s^i
t^{5-i}+16\epsilon(s+t)$. Let us balance this inequality: we
choose $s = \epsilon^{\frac16}$, $t= \epsilon^{\frac12}$ (we
assume that $\epsilon$ was chosen sufficiently small so that the
last inequality is valid for these $s$ and $t$) and get that
$$\Delta \leq \text{const}(F,G) \cdot \epsilon^{\frac13}\;,$$
as required. \qed

\medskip
\noindent
\begin{rem}\label{rem-sharpness}{\rm Let us make the following manipulation with
inequality \eqref{eq-rate-vsp-2}: put $s=t$, divide by $t^3$ and
rewrite the inequality as follows:
\begin{equation}\label{eq-remark}
\max (A+B) -\max A' -\max B' \leq 4Et^2 +
16\frac{\epsilon}{t^2}\;.
\end{equation}
At first glance this is not too encouraging since we have to
estimate from above the quantity $ \max A+ \max B -\max A' -\max
B'\;,$ while in general $\max A+ \max B$ is {\it greater} than
$\max (A+B)$. As we have seen in the proof above, we bypassed this
difficulty by letting $s$ and $t$ to have different asymptotical
behavior in $\epsilon$. It might well happen that at this point we
lost sharpness in our bound $\sim \epsilon^{1/3}$ for the
convergence rate given in Theorem~\ref{thm-rate}, cf.
Question~\ref{quest-power} above. Here is some evidence in favor
of this possibility: assume for a moment that $\max A+ \max B =
\max (A+B)$. Putting $t = \epsilon^{1/4}$ in inequality
\eqref{eq-remark} we get that the right hand side is of the order
$\sim \epsilon^{1/2}$. At the same time it is easy to exhibit
examples of functions $F,G$ on the 2-sphere with, in notations of
Theorem~\ref{thm-rate},
$$\Phi(F,G) - \overline{\Phi}_{\epsilon}(F,G)\leq
C \cdot\epsilon^{\frac12}\;.$$ Thus $\sim \epsilon^{1/2}$ could be
considered as another candidate for the sharp power law in the
convergence rate.}
\end{rem}

\section{Decreasing
$\max \{\{F,G\},F\}$, $-\min\{\{F,G\},F\}$ }
\label{sec-oscFGG-example}

$\;$ \noindent In this section we prove the following result.

\medskip
\noindent
\begin{thm}\label{thm-no-semicontinuity}
On every symplectic manifold there exists a collection of functions
$F_N,F,G$ so that $F_N\stackrel{C^0}{\longrightarrow}F$ and
$$\liminf_{N \to \infty}\max\{\{F_N,G\},F_N\} < 0.99 \cdot\max
\{\{F,G\},F\}\;,$$
$$\liminf_{N \to \infty}(-\min\{\{F_N,G\},F_N\}) < -0.99\cdot \min
\{\{F,G\},F\}\;.$$
\end{thm}

\medskip
\noindent {\sc Beginning the construction:}  To make the example
more transparent we first describe it in the two-dimensional case.
Consider $M=\R^2$ with the standard symplectic form $\omega=dp\wedge
dq$. Recall that the Poisson bracket is defined by
$\{F,G\}=dF(\text{sgrad}\;G)$, so $\{p,q\} = -1$. We first look for
$F,G,F_N: \R^2\to \R$  in the form
$$F=u(p),\;G=-v(q),\; F_N = u(p) + \frac{1}{N} a(q)\sin Nu(p)\;.$$
Here $u,v,a$ are compactly supported in $\R$, thus at this stage
$F,G,F_N$ do {\bf not} have compact supports in $\R^2$ yet. This
will be corrected later. The choice of $u(p)$ is essentially
arbitrary, while $v(q),a(q)$ will be chosen in a special way
below.

\medskip
\noindent Observe that
$$ \{\{F,G\},F\} = u'(p)^2 v''(q)\;,$$
so $$\max \{\{F,G\}, F\}=\max u'(p)^2 \max v''(q)\;.$$ Furthermore,
a straightforward but lengthy calculation shows that
$$\{ \{F_N,G\}, F_N\} = u'(p)^2 R(p,q)+ O(\frac{1}{N}),$$
where
$$R(p,q) = v''(q) (a(q)\cos Nu(p) +1)^2 + a'(q)v'(q)(a(q)+\cos N
u(p))\;.$$ Write $v'(q) = w(q)$ so that
\begin{equation}\label{eq-R}
R(p,q) = w'(q) (a(q)\cos Nu(p) +1)^2 + a'(q)w(q)(a(q)+\cos N
u(p))\;. \end{equation}

\medskip
\noindent Introduce the function
$$r(\alpha,\gamma,z):= (\alpha z+1)^2 -\gamma (\alpha+z)\;.$$
With this notation
$$R(p,q) = w'(q)r\bigg(a(q),-a'(q)w(q)/w'(q),\cos Nu(p)\bigg)$$
whenever $w'(q) \neq 0$.

\medskip
\noindent
\begin{lemma}\label{lem-r} For the specific fixed values $\alpha = 1.1$ and
$\gamma=1.63$ the function $z\mapsto r_{\alpha, \gamma} (z) := r
(\alpha, \gamma, z)$ satisfies the following inequality:
$$-0.99 < r_{\alpha,\gamma} (z) < 0.99  \;\; \forall z \in
[-1,1]\;.$$
\end{lemma}

\medskip
\noindent
\begin{proof} Write
$$r(z):=r_{1.1, 1.63} (z)= 1.21 z^2 + 0.57 z - 0.793\;.$$
We have to check the values of $r$ at the endpoints $\pm 1$ and at
the critical point. We have $r(-1)= -0.153, r(1) = 0.987$. The
critical value equals $-0.57^2/(4 \cdot 1.21) - 0.793 \approx
-0.86$. We conclude that $|r(z)| < 0.99$ for $z \in [-1,1]$.
\end{proof}

\medskip
\noindent Now we are ready to describe our example. Fix real
numbers $c_1 <c_2<c_3<c_4$ so that $c_{i+1}-c_i=\delta$, where
$\delta>0 $ will play the role of a small parameter in our
construction.  Fix $\kappa>0$ so that the inequality in
Lemma~\ref{lem-r} holds for the function $r_{\alpha, 1.63} (z)$
(over $[-1,1]$) for every $\alpha \in [1.1-\kappa,1.1+\kappa]$.
The compactly supported functions $w$ and $a$ are chosen as
follows:

\medskip\noindent{\sc Conditions on $w(q)$:}

\medskip
\noindent (i) $\max_{[c_2,c_3]} w' = 1$,  $\min_{[c_2,c_3]} w' =
-1$,  $w'(c_2)= 0.001$, $w'(c_3)=-0.001$.

\medskip
\noindent (ii) $ 1 \leq w(q) \leq 2$ for $q \in [c_1,c_4]$.

\medskip
\noindent (iii) $|w'(q)| \leq 0.01$ for $q \in \R_+ \setminus
[c_2, c_3]$.

\medskip
\noindent (iv)
 $|w(q)| \leq 3$ for $q \in
\R_+ \setminus [c_1, c_4]$.

\medskip
\noindent (v) $\max w = -\min w =1$.

\medskip
\noindent (vi) $\int_{-\infty}^{+\infty} w(q)\;dq=0$.

\bigskip

\medskip\noindent{\sc Conditions on $a(q)$:}

\medskip
\noindent (i) $a'(q) = -1.63 w'(q)/w(q)$ for $q \in [c_1,c_4]$.

\medskip
\noindent  (ii) $ a(q) \in [1.1-\kappa,1.1+\kappa] $ for $q \in
[c_1,c_4]$.

\medskip
\noindent (iii) $|a'(q)| \leq 0.03 $ and $|a(q)| \leq 2 $ for $q
\in \R_+ \setminus [c_1, c_4]$.

\medskip
\noindent Assumption (iii) on $a$ is compatible with (i) since on
$[c_1,c_2] \cup [c_3,c_4]$ we have $|a'(q)| \leq 2\cdot 0.01:1 =
0.02$. Assumption (ii) is achieved by taking $\delta$ as small as
needed.

\medskip
\noindent {\sc The bound on $R(p,q)$:} Let us check  that
\begin{equation}\label{eq-R-vsp1}
|R(p,q)| \leq 0.99\cdot \max w'= -0.99\cdot \min w' = 0.99\;.
\end{equation}

\medskip
\noindent {\sc Case 1:} For $q \in [c_1,c_4]$ we have
$$|R(p,q)| = |w'(q)| \cdot |r(\alpha, 1.63, z)|,$$
with $ \alpha \in [1.1-\kappa,1.1+\kappa] $ and $z \in [-1,1]$. By
Lemma~\ref{lem-r} and due to the choice of $\kappa$, we have
$|R(p,q)| \leq 0.99 \cdot |w'(q)| \leq 0.99$.

\medskip
\noindent{\sc Case 2:} For $q \in \R_+ \setminus [c_1,c_4]$ we
have
$$-0.99< -0.36= -0.01\cdot (2+1)^2 - 0.03\cdot 3\cdot (2+1) \leq$$
$$\leq R(p,q) \leq 0.01 \cdot (2+1)^2 + 0.03\cdot 3 \cdot (2+1) = 0.36 <
0.99\;.$$ This completes the proof of \eqref{eq-R-vsp1}.

\medskip
\noindent{\sc Summary:} Since $ \{\{F,G\},F\} = u'(p)^2 w'(q)$ and
$$\{\{F_N,G\}, F_N\} = u'(p)^2 R(p,q)+ O(\frac{1}{N})\;,$$ we conclude
that the functions $F,F_N, G$ satisfy inequalities in Theorem
\ref{thm-no-semicontinuity}.

\medskip
\noindent{\sc Making the functions compactly supported in $\R^2$:}
According to our construction, the support of the function $u=u(p)$,
as a function on $\R$, is contained in some closed interval $I$ and
the supports of $v=v(q)$ and $a=a(q)$, as functions on $\R$, are
both contained in some closed interval $J$. Hence the supports of
$F$ and $F_N$ in $\R^2$ are contained in $I\times \R$ and the
support of $G$ is contained in $\R\times J$.

Let us choose a cut-off function $\phi: \R^2\to [0,1]$ which is
equal to $1$ on $I\times J$. Then
$$ \{ \phi F, \phi G\} = \phi^2 \{ F,G\},$$
since $\phi$ is constant on $\supp F\, \cap\ \supp G$. For the same
reason, since \break $\supp \{ F, G\} \subset I\times J$ and $\phi$
is constant on $\supp F\, \cap\, \supp \{ F, G\}$, we have
$$\{ \phi F,\{ \phi F, \phi G\}\} = \phi^3 \{ F, \{ F, G\}\}.$$
Since $\supp \{ F, \{ F, G\}\} \subset \supp F\, \cap\, \supp \{ F,
G\} \subset I\times J$ we get that
$$\max \{ \phi F, \{ \phi F, \phi G\}\} = \max \{ F, \{ F, G\}\}.$$

Replacing $F$ by $F_N$ and noticing that the previous considerations
depend only on the supports of $F$ and $G$, we get that
$$\{ \phi F_N,\{ \phi F_N, \phi G\}\} = \phi^3 \{ F_N, \{ F_N, G\}\}$$
and
$$\max \{ \phi F_N, \{ \phi F_N, \phi G\}\} = \max \{ F_N, \{ F_N, G\}\}$$
for the same reason as above.

The same equalities holds for the minima. Thus the functions
 $\tF:=\phi F, \tF_N:=\phi F_N, \tG:=\phi G$ are
compactly supported in $\R^2$ and satisfy the inequalities in
Theorem~\ref{thm-no-semicontinuity}.

\medskip
\noindent {\sc Implanting into a symplectic manifold:} Now the
example can be easily generalized to a higher-dimensional Darboux
chart (and hence implanted into any symplectic manifold), cf.
\cite{EP-Poisson1}, proof of Theorem 1.6.

Namely, assume ${\rm dim}\, M = 2n > 2$. In a local Darboux chart
with coordinates $p_1, q_1,\ldots, p_n, q_n$ on $M$ choose an open
cube $$P= K^{2n-2} \times K^2,$$ where $K^{2n-2}$ is an open cube
in the $(p_1, q_1,\ldots, p_{n-1}, q_{n-1})$-coordinate plane and
$K^2$ is a open square in the $(p_n, q_n)$-coordinate plane. Fix a
smooth compactly supported function $\chi: K^{2n-2}\to [0,1]$
which reaches the value $1$ at some point. Given a smooth
compactly supported function $L$ on $K^2$, define the function
$\chi L\in C^\infty_c (M)$ as
$$\chi L (p_1, q_1,\ldots, p_n, q_n) := \chi (p_1, q_1,\ldots, p_{n-1}, q_{n-1}) L (p_n, q_n)$$
on $P$ and as zero outside $P$.

Now pick functions $\tF,\tG,\tF_N\in C^\infty_c (K^2)$ as above. Set
$$F':= \chi \tF, G':=\chi \tG, F'_N := \chi \tF_N \in C^\infty_c (P)\subset C^\infty_c (M).$$
Clearly $F'_N$ converges uniformly to $F'$. It is also clear that
$$\{ \{ F', G'\}, F'\} = \{ \{ \chi \tF, \chi \tG\}, \chi \tF\}
= \chi^3 \{\{ \tF, \tG\},  \tF\},$$
$$\{ \{ F'_N, G'\}, F'_N\} = \{ \{ \chi \tF_N, \chi \tG\}, \chi \tF_N\}
= \chi^3 \{\{ \tF_N, \tG\},  \tF_N\},$$ because the Poisson bracket
of $\chi$ and any function of $p_n, q_n$ vanishes identically.
Recalling how $\chi$ was chosen we see that
$$\max \{ \{ F', G'\}, F'\} = \max \{ \{ \tF, \tG\}, \tF\},$$
$$\max \{ \{ F'_N, G'\}, F'_N\} = \max \{ \{ \tF_N, \tG\}, \tF_N\}.$$
Hence $F'$, $G'$, $F'_N$ satisfy the required properties (since
$\tF$, $\tG$, $\tF_N$ do). This completes the proof of the theorem.
\Qed

\medskip
\noindent\begin{rem}{\rm  It would be interesting to find out how
small can the ratio
\[
\frac{\displaystyle\liminf_{F',G'\stackrel{C^0}{\longrightarrow}F,G}
\osc\, \{\{ F',G'\}, F'\} }{\osc\, \{\{ F,G\}, F\} } \]
 be made by varying $F$
and $G$ so that $\osc\, \{\{F,G\},F\}\neq 0$. Clearly, this ratio
always belongs to $(0,1]$. In the example constructed above it is no
bigger than 0.99. In fact, one can slightly modify that example to
show that the ratio
$$\frac{\displaystyle\liminf_{F',G'\stackrel{C^0}{\longrightarrow}F,G} \max\,
\{\{F',G'\},F'\}}{\max\, \{\{F,G\},F\}}$$ can be made arbitrarily
small by an appropriate choice of $F$ and $G$.}
\end{rem}

\section{The proof of the Dichotomy
Theorem: conclusion}\label{sec-end}

We shall write $\R^4_+$ for the non-negative orthant of $\R^4$.
The functionals $\Phi^v(F,G)$ respect the natural actions of the
dihedral group $D_4$ on vectors $v \in \R^4_+$ and the variables
$(F,G)$. In particular, define linear transformations $A,B,C$ of
$\R^4$ by
$$A(v_1,v_2,v_3,v_4) = (v_2,v_1,v_3,v_4)\;,$$
$$B(v_1,v_2,v_3,v_4) = (v_1,v_2,v_4,v_3)\;,$$
$$C(v_1,v_2,v_3,v_4) = (v_3,v_4,v_1,v_2)\;$$
which generate the $D_4$-action on $\R^4_+$. Then $$\Phi^v (F,-G) =
\Phi^{Av}(F,G),$$ $$\Phi^v(-F,G)=\Phi^{Bv}(F,G),$$
$$\Phi^v(-G,-F)=
\Phi^{Cv}(F,G)\;.$$

Next, $\Phi^v$ obeys the following scaling laws: given
$\alpha,\beta>0$, we have $\Phi^v(\alpha F,\beta G) = \Phi^w(F,G)$
with $$w=  (\alpha^2\beta v_1, \alpha\beta^2 v_2,\alpha^2\beta
v_3, \alpha\beta^2 v_4)\;.$$

Note now that all the properties of functionals $\Phi^v$ appearing
in the Dichotomy Theorem are invariant under the action of the
dihedral group and rescaling.

\medskip
\noindent

\medskip
\noindent {\bf Proof of Theorem~\ref{thm-main-1}(i)} Assume that
either $v_3=v_4=0$ or $v_1=v_2=0$. Applying if necessary  the
dihedral group we can assume without loss of generality that
$$\Phi^v(F,G) = v_1\max\{\{F,G\},F\}-v_2\min\{\{F,G\},F\}\;, v_1 > 0,v_2 \geq 0\;.$$
By Theorem~\ref{thm-no-semicontinuity} $\Phi^v$ is not lower
semicontinuous. Further,  $$\Phi^v(F,G) \geq
v_1\max\{\{F,G\},F\}\;,$$ and hence $\Phi^v$ is weakly robust by
inequality \eqref{eqn-double-bracket-liminf-osc}.  \qed

\medskip
\noindent {\bf Proof of Theorem~\ref{thm-main-1}(ii)} Put
$$\mu_+(F,G)=\max\{\{F,G\},F\},\;\mu_-(F,G) =
-\min\{\{F,G\},F\}\;,$$
$$\nu_+(F,G)=\max\{\{F,G\},G\},\;\nu_-(F,G) =
-\min\{\{F,G\},G\}\;.$$ With this notation $\mu_+(F,G)+\nu_+(F,G)$
is lower semicontinuous by Theorem~\ref{thm-rate}. We shall need
the following auxiliary result.

\medskip
\noindent\begin{lemma}\label{lem-vsp-iii} Let
$(F_i,G_i)\stackrel{C^0}{\longrightarrow}(F,G)$. If $\mu_+(F_i,G_i)
\leq K <\infty$ for all $i$
\begin{equation}\label{eq-vsp-iii-minus}
\liminf_{i\to \infty} \nu_-(F_i,G_i) \geq \nu_-(F,G)\;.
\end{equation}
Similarly, if $\nu_+(F_i,G_i) \leq K <\infty$ for all $i$
\begin{equation}\label{eq-vsp-iii-minus-mu}
\liminf_{i\to \infty} \mu_-(F_i,G_i) \geq \mu_-(F,G)\;.
\end{equation}

\end{lemma}

\begin{proof} Assume that $\mu_+(F_i,G_i)
\leq K <\infty$ for all $i$. First we prove that
\begin{equation}\label{eq-vsp-iii-plus}
\liminf_{i\to \infty} \nu_+(F_i,G_i) \geq \nu_+(F,G)\;.
\end{equation}
Assume on the contrary that along a subsequence
$$\lim_{i \to \infty} \nu_+(F_i,G_i) \leq E<  \nu_+(F,G)\;.$$
Pick  $0< \alpha < (\nu_+(F,G)-E)/K$. Then for $i$ sufficiently
large
$$\mu_+(\alpha F_i,G_i) + \nu_+(\alpha F_i,G_i)= \alpha^2\mu_+(F_i,G_i) +\alpha \nu_+ (F_i,G_i)
\leq  \alpha (\alpha K+ E)$$ $$ < \nu_+(\alpha F, G) <
\mu_+(\alpha F,G) +\nu_+(\alpha F,G)\;,$$ which contradicts lower
semicontinuity of $\mu_+ + \nu_+$ at $(\alpha F,G)$.  This proves
\eqref{eq-vsp-iii-plus}. Replacing $F$ by $-F$, we obtain the
desired inequality \eqref{eq-vsp-iii-minus} from
\eqref{eq-vsp-iii-plus}. The proof of \eqref{eq-vsp-iii-minus-mu}
is analogous.
\end{proof}

\medskip
\noindent Now we are ready to complete the proof of  Theorem
\ref{thm-main-1}(ii). Let $v \in \R^4_+$ be such that at least one
of $v_1,v_2$ is positive and at least one of $v_3,v_4$ is
positive. Applying the action of the dihedral group and rescaling,
we can achieve that $v_1=v_3=1$, and thus, without loss of
generality,
$$\Phi^v = \mu_+ + \nu_+ + v_2 \cdot \mu_- + v_4 \cdot \nu_-\;.$$
Assume that $(F_i,G_i)\stackrel{C^0}{\longrightarrow}(F,G)$ so
that $\Phi^v(F_i,G_i)$ is bounded. The lower semicontinuity of
$\mu_+ +\nu_+$ and Lemma~\ref{lem-vsp-iii} yield
$$\liminf_{i \to \infty} \Phi^v (F_i,G_i) \geq \Phi^v(F,G)\;.$$
This finishes off the proof. \qed

\section{Higher iterated brackets: discussion}\label{sec-disc}

Denote by $\cal{P}_N$, $N \in \N$, the set of all Lie monomials in
two variables involving $N$-times-iterated Poisson brackets. Given
monomials $p_1,...,p_d \in \cal{P}_N$ and non-negative numbers
$\alpha_j, \beta_j \geq 0$, consider a functional $\Phi(F,G)$,
given by
\begin{equation}\label{eq-discus-1}
\sum_{j=1}^d \alpha_j \cdot \max p_j(F,G) -\beta_j \cdot \min
p_j(F,G)\;. \end{equation}
 In the case $N \geq 3$ the problem of
detecting whether $\Phi$ is weakly robust or lower semicontinuous is
at the moment almost completely out of reach. The simplest case
where the answer is unknown to us is $N=3,d=1, p(F,G) =
\{\{\{F,G\},F\},G\}$.

To emphasize the main difficulty, let us recall that our strategy
of proving the lower semi-continuity in the case of the ordinary
bracket and the double bracket is as follows: We design an
expression of the form $u_{t} = \prod_j \phi^{t}_{a_j F + b_j G}$
so that the Hofer's (semi)norm $\rho(u_t)$ admits ``tight" lower
and upper bounds in terms of the maxima/minima of Lie polynomials
involving monomials $p_j$ entering $\Phi$. For instance, for
$\Phi(F,G) = \{F,G\}$ we use the flow $u_t =
[\phi^{-t}_F,\phi^{-t}_G]$ and for $\Phi(F,G)=\max\{\{F,G\},F\} +
\max\{\{F,G\},G\}\;$ we use the flow $u_t = [\phi^t_{-F}
\phi^t_{-G}, \phi^t_{F+G}]\;.$ Combining the above-mentioned lower
and upper bounds, we obtain an inequality involving iterated
Poisson brackets, which eventually yields the desired
semicontinuity.

For a general functional of the form \eqref{eq-discus-1} it is
unclear how to design expressions $u_t$ as above leading to
``tight" lower and upper bounds for $\rho(u_t)$, and it is even
unclear whether such the expressions do exist at all. Thus new
ideas are needed.

\medskip
\noindent
\begin{question}\label{quest-1} Is the functional
$$\Phi_N(F,G) = \sum_{p \in \cal{P}_N} \osc\,p(F,G)$$ lower
semicontinuous?
\end{question}

\medskip
\noindent  Note that the answer is affirmative for $N=1,2$.

\medskip

As far as the weak robustness is concerned, we are able to settle
a particular case which is a direct generalization of
Proposition~\ref{thm-double-bracket}. Denote by $\text{ad}_F:
C^{\infty}_c (M) \to C^{\infty}_c (M)$ the operator $G \mapsto
\{G,F\}$.

\begin{prop}\label{prop-iter} For every $N \in \N$ the functional
$\Phi(F,G) = \osc\,(\text{ad}_F)^N G $ is weakly robust.
\end{prop}

\medskip
\noindent \begin{proof} Indeed, by Kolmogorov's generalization of
the Landau-Hadamard inequality \cite{Kolm1, Kolm2}, cf.
\cite{Mitrinovic-et-al}, there exists a constant $C_N
>0$ so that for every smooth function $v \in C^{\infty}(\R)$, which
is bounded with its $N$ derivatives and does not vanish identically,
one has the following lower bound on the uniform norm of the $N$-th
derivative $v^{(N)}$ of $v$ in terms of its first derivative $v'$:
$$\osc\,v^{(N)} \geq C_N \cdot \frac{||v'||^N }{||v||^{N-1}}\;.$$
Choose a point $x$ in the symplectic manifold $M$ where the uniform
norm of $\{G,F\}$ is attained. Applying the Kolmogorov inequality to
the function $v(t)= G(\phi_F^t(x))$ we get that
\begin{equation}
\label{eqn-mult-brackets-robust} \Phi(F,G) \geq C_N
\frac{||\{F,G\}||^N}{||G||^{N-1}}\;.
\end{equation}
 Since the functional $\{F,G\}$
is lower semicontinuous, the above inequality readily yields the
weak robustness of $\Phi$. \end{proof}

\medskip

Inequality \eqref{eqn-mult-brackets-robust} can be iterated as
follows. For integers $m \geq 1, k \geq 0$ introduce the functional
$$\Phi_{k,m}(F,G) = \osc\,(\text{ad}_H)^m G\;,\text{where}\;\;H=
(\text{ad}_G)^k
F\;.$$

\begin{prop}\label{prop-iter-1} There exist constants $C_{k,m}> 0$ so
that
$$\Phi_{k,m} (F,G) \geq C_{k,m} \cdot
\frac{||\{F,G\}||^{(k+1)m}}{||F||^{km}||G||^{m-1}}\;.$$ In
particular, $\Phi_{k,m}$ is weakly robust.
\end{prop}

\medskip
\noindent \begin{proof} Applying \eqref{eqn-mult-brackets-robust}
twice we get
$$\Phi_{k,m} (F,G) \geq \text{const}\cdot
\frac{||\{H,G\}||^m}{||G||^{m-1}} \geq \text{const}\cdot
\frac{||(\text{ad}_G)^{k+1}F||^m}{||G||^{m-1}}
$$
$$
\geq \text{const}\cdot
\frac{||\{F,G\}||^{(k+1)m}}{||F||^{km}||G||^{m-1}}\;.$$
\end{proof}

\medskip
\noindent
\begin{problem} Prove that for $N \geq 3$ the functional $\osc\,(\text{ad}_F)^N G$
 is {\bf not} lower semicontinuous.
\end{problem}

\medskip
\noindent In Section~\ref{sec-oscFGG-example} we proved this for
$N =2$. It is not clear whether our method extends to $N \geq 3$.
A natural generalization of this problem would be to detect the
failure of lower semicontinuity for the functionals $\Phi_{k,m}$
from Proposition~\ref{prop-iter-1} (except the trivial case
$k=0,m=1$).

\bigskip \noindent {\bf Acknowledgements.} We thank Vincent Humiliere
for a stimulating correspondence. We thank Lev Buhovsky for
numerous useful discussions, Misha Sodin for illuminating
consultations on the Landau-Hadamard-Kolmo\-go\-rov inequalities
and the anonymous referee for helpful comments.

\bibliographystyle{alpha}

\bigskip

\noindent
\begin{tabular}{ll}
Michael Entov & Leonid Polterovich\\
Department of Mathematics & School of Mathematical Sciences\\
Technion - Israel Inst. of Technology & Tel Aviv University\\
Haifa 32000, Israel & Tel Aviv 69978, Israel\\
entov@math.technion.ac.il & polterov@post.tau.ac.il\\
\end{tabular}

\end{document}